\documentclass[graybox]{svmult}

\usepackage{type1cm}        % activate if the above 3 fonts are
\usepackage{makeidx}         % allows index generation
\usepackage{graphicx}        % standard LaTeX graphics tool
\usepackage{multicol}        % used for the two-column index
\usepackage[bottom]{footmisc}% places footnotes at page bottom

\usepackage{newtxtext}       % 
\usepackage[varvw]{newtxmath}       % selects Times Roman as basic font

\makeindex             % used for the subject index
\usepackage{subcaption}
\usepackage{bm} 
\usepackage{soul,xcolor}
\setstcolor{red}
\begin{document}

\title*{A Block Jacobi Sweeping Preconditioner for the Helmholtz Equation}
\author{Ruiyang Dai}
\institute{Ruiyang Dai \at Laboratory J.L. Lions, Sorbonne Universit\'e, Paris, France. \email{ruiyang.dai@upmc.fr}}
\maketitle

%\abstract*{Each chapter should be preceded by an abstract (no more than 200 words) that summarizes the content. The abstract will appear \textit{online} at \url{www.SpringerLink.com} and be available with unrestricted access. This allows unregistered users to read the abstract as a teaser for the complete chapter.
%Please use the 'starred' version of the \texttt{abstract} command for typesetting the text of the online abstracts (cf. source file of this chapter template \texttt{abstract}) and include them with the source files of your manuscript. Use the plain \texttt{abstract} command if the abstract is also to appear in the printed version of the book.}

%\abstract{Each chapter should be preceded by an abstract (no more than 200 words) that summarizes the content. The abstract will appear \textit{online} at \url{www.SpringerLink.com} and be available with unrestricted access. This allows unregistered users to read the abstract as a teaser for the complete chapter.\newline\indent
%Please use the 'starred' version of the \texttt{abstract} command for typesetting the text of the online abstracts (cf. source file of this chapter template \texttt{abstract}) and include them with the source files of your manuscript. Use the plain \texttt{abstract} command if the abstract is also to appear in the printed version of the book.}

\section{Introduction}
\label{sec:introduction}

Solving Helmholtz problems using numerical methods is challenging due to the large, indefinite, and ill-conditioned linear systems that result, which cannot be solved using classical direct or iterative solvers \cite{ernst2011difficult}.
While optimized Schwarz (OS) methods have been proposed as an alternative, the number of iterations required by Krylov methods increases with the number of subdomains, especially for layer-type domain decompositions \cite{nataf1994optimal}.
Preconditioners, such as sweeping preconditioners, are necessary when using iterative methods to solve Helmholtz problems. 
Lately, there has been significant interest in sweeping preconditioners, invented by \cite{engquist2011sweeping-moving,engquist2011sweeping}, that achieve quasi-linear asymptotic complexity.
Despite their effectiveness, sweeping preconditioners face challenges with parallel scalability due to the inherently sequential nature of their operations, as well as the need to ensure accurate and consistent information transfer between subdomains. 
These challenges can restrict the use of layer-type domain decompositions.

%Solving Helmholtz problems using grid-based numerical methods is a tough challenge 
%because the resulting linear systems are huge, indefinite, and ill-conditioned, 
%which cannot be handled by employing classical direct or iterative solvers \cite{ernst2011difficult}.
%Although optimized Schwarz (OS) methods, a family of Domain Decomposition (DD) methods, 
%are proposed as an alternative,
%the number of iterations of the linear system accelerated by the Krylov
%method will grow as the number of subdomains increases, especially for layer-type
%domain decompositions \cite{nataf1994optimal}.
%When using iterative methods to solve Helmholtz problems, it is necessary to apply preconditioners.
%Lately, there has been significant interest in sweeping preconditioners, invented by \cite{engquist2011sweeping-moving,engquist2011sweeping}, that achieve quasi-linear asymptotic complexity.
%Despite their effectiveness, sweeping preconditioners face challenges with parallel scalability due to the inherently sequential nature of their operations, as well as the need to ensure accurate and consistent information transfer between subdomains. 
%These challenges can restrict the use of layer-type domain decompositions.

To address these challenges, recent research has focused on improving parallel performance through new sweeping strategies on checkerboard domain decompositions that can handle more general domain decompositions. 
Several sweeping algorithms have been proposed that improve parallelism by ensuring consistent transfer among subdomains, such as L-sweeps preconditioners \cite{taus2020sweeps}, trace transfer-based diagonal sweeping preconditioners \cite{leng2022trace}, and multidirectional sweeping preconditioners \cite{dai2022multidirectionnal}, with high-order transmission conditions and cross-point treatments \cite{modave2020non}.

%To address this issue, it is necessary to consider more general domain decompositions. 
%Recently, 
%the parallel performances of sweeping-type algorithms have been improved by departing from standard layer-type domain decomposition 
%and introducing a new sweeping strategy on checkerboard-type domain decomposition, 
%such as L-sweeps preconditioners \cite{taus2020sweeps} and trace transfer-based diagonal sweeping preconditioners \cite{leng2022trace} for the Helmholtz equation, 
%which are proposed in the context of the method of polarized traces and the source transfer method, respectively, 
%and multidirectional sweeping preconditioners \cite{dai2022multidirectionnal} for the OS methods applied to the Helmholtz equation, with high-order transmission conditions and cross-point treatments \cite{modave2020non}. 
%The key to improving the parallelism of these algorithms is a consistent information transfer among subdomains. 
%Therefore, several processes can be performed concurrently. 

Subdomains in sweeping algorithms can be assigned to Message Passing Interface (MPI) ranks based on rows or columns. 
This enables parallel application of sweeping algorithms for a single right-hand side. 
However, these approaches still have limitations, including long preconditioning procedures, waste of computation resources, and relatively high computation costs.
To overcome these limitations, the authors propose a block Jacobi sweeping preconditioner that uses block Jacobi matrices to decompose full sweeps into several partial sweeps, 
which can be thought of as sweeps that operate on a subset of the subdomains.
These partial sweeps can be performed concurrently. This approach enhances scalability and makes full use of resources on parallel computer architectures.

%In these sweeping algorithms, 
%there is a strategy proposed to assign subdomains to Message Passing Interface (MPI) ranks: 
%row- or column-based assignments of subdomains to MPI ranks (each row or column of subdomains is assigned to one MPI rank).
%With this strategy, sweeping algorithms can be applied optimally in parallel for a single right-hand side.
%While the algorithms with row- or column-based assignments possess much better parallel performances compared to 
%the layer-by-layer sweeping algorithm, there are still some limitations. 
%Firstly, each step of sweeps is still a sequential process, which leads to a long preconditioning procedure compared to 
%the matrix-free-product phase. Next, the starting point of sweeping preconditioning is from one subdomain which causes some 
%waste of computation resources. 
%At last, the current sweeping approaches can be interpreted as a completely approximate $LU$ factorization \cite{vion2014double},
%which implies a relatively high computation cost.
%In this work, we propose a block Jacobi sweeping preconditioner, which are improved variants of sweeping-type preconditioners. Thanks to the introduction of the block Jacobi matrices, full sweeps are decomposed into several partial sweeps, which can be performed concurrently.
%This novel sweeping preconditioner provide enhanced scalability and make full usage of the resources on parallel computer architectures.

\section{Notations}
\label{sec:notations}

Let $\mathbf{i} = (i_1,i_2) \in \mathbb{N}^2$ be a multi-index denoting the subdomain number.
We define the discrete $l^1$ norm by:
%\begin{equation}
%|\mathbf{i}|_1 := |i_1| + |i_2|.
%\end{equation}
$
|\mathbf{i}|_1 := |i_1| + |i_2|.
$
We use the convention that two multi-indices $\mathbf{i}$ and $\mathbf{j}$ are equal if and only if $i_1=j_1$ and $i_2=j_2$.
%and the following order on multi-indices, which is denoted by $\mathbf{i} < \mathbf{j}$
%if and only if
%\begin{enumerate}
%\item $|\mathbf{i}|_1 < |\mathbf{j}|_1$ or,
%\item $|\mathbf{i}|_1 = |\mathbf{j}|_1$ and $i_1 < j_1$.
%\end{enumerate}
\begin{definition}
The \textit{lexicographic order} on multi-indices is 
the relation defined by $\mathbf{i} < \mathbf{j}$ if and only if
%\begin{enumerate}
%\item $|\mathbf{i}|_1 < |\mathbf{j}|_1$, or
%\item $|\mathbf{i}|_1 = |\mathbf{j}|_1$, and $i_1 < j_1$.
%\end{enumerate}
$|\mathbf{i}|_1 < |\mathbf{j}|_1$, or $|\mathbf{i}|_1 = |\mathbf{j}|_1$, and $i_1 < j_1$.
\end{definition}

\begin{definition}
The \textit{lexicographic order} on pair multi-indices $(\mathbf{i}, \mathbf{j}) \in \mathbb{N}^2 \times \mathbb{N}^2$ is 
the relation defined by $(\mathbf{i}, \mathbf{j}) < (\mathbf{k}, \mathbf{l})$
if and only if
%\begin{enumerate}
%\item $\mathbf{i} < \mathbf{k}$, or
%\item $i_1 = k_1$, $i_2 = k_2$, and $\mathbf{j} < \mathbf{l}$.
%\end{enumerate}
$\mathbf{i} < \mathbf{k}$, or $\mathbf{i} = \mathbf{k}$, and $\mathbf{j} < \mathbf{l}$.
\end{definition}

We define a function $m$ maps pair multi-indices in the lexicographic order 
to the natural numbers in a monotonically increasing fashion 
$m:\mathbb{N}^2 \times \mathbb{N}^2 \rightarrow \mathbb{N}$, 
such that $m((1,1), (1,2)) = 1$, $m((1,1), (2,1)) = 2$, $m((1,2), (1,1)) = 3$, etc. 

We consider $\Omega \subset \mathbb{R}^2$ be a square domain with boundary $\partial \Omega$, 
which is given by the intersection of the scattered boundary $\partial\Omega^{\text{sca}}$ 
with the external artificial boundaries $\Gamma_i^{\infty}$ for $i=1,2,3,4$,
%We consider a square domain $\Omega \in \mathbb{R}^2$ with boundary $\partial\Omega = \partial\Omega^{\text{sca}}\cap\Gamma_i^{\infty}$ 
%($\Gamma_i^{\infty}$, $i = 1,2,3,4$ are the external artificial boundaries of $\Omega$, and $\Omega^{\text{sca}}$ a sound-soft obstacle), 
and its a non-overlapping checkerboard partition, which consists in a lattice of rectangular non-overlapping subdomains 
$\Omega_{\mathbf{i}}$ with $N_1$ columns and $N_2$ rows ($i_1 = 1,\ldots, N_1$, and $i_2 = 1,\ldots, N_2$), 
that is
$$
\overline{\Omega} = \bigcup
\overline{\Omega}_{\mathbf{i}}, \quad \text{and} \quad
\Omega_{\mathbf{i}} \cap \Omega_{\mathbf{j}} = \varnothing \quad \text{for} \quad \mathbf{j} \neq \mathbf{i}.
$$
And we say that $\exists \mathbf{i}$, such that $\Omega^{\text{sca}}\subseteq \Omega^\circ_{\mathbf{i}}$ and $\partial\Omega^{\text{sca}} \cap \partial\Omega_{\mathbf{i}} = \varnothing$. 
The boundary of a subdomain $\Omega_{\mathbf{i}}$ is split into two parts: the exterior part $\partial \Omega_{\mathbf{i}} \cap \Gamma^{\infty}_i$
and the interior part including decomposed interior interfaces
$\Gamma_{\mathbf{i},\mathbf{j}} := \partial\Omega_{\mathbf{i}} \cap \partial\Omega_{\mathbf{j}}$ ($\mathbf{j} \neq \mathbf{i}$),
and $\Gamma_{\mathbf{i},\mathbf{j}} = \Gamma_{\mathbf{j},\mathbf{i}}$.
There are $N_{\text{dom}} = N_1 \times N_2$ subdomains, $N_e = 2N_1N_2-N_1-N_2$ interior interfaces. We define the number of diagonal groups $N_g := N_1+N_2-1$.

\section{Non-overlapping domain decomposition method}
\label{sec:non_overlapping}

%We consider the two-dimensional Helmholtz equation in $\Omega$
%with an absorbing boundary condition (ABC) on the boundary $\Gamma^{\infty}$.
%We seek the scattered field $u(\mathbf{x})$ that verifies
We study the 2D Helmholtz equation in $\Omega$ with an absorbing boundary condition on $\Gamma^{\infty}_i$.
We seek the field $u(\mathbf{x})$ that verifies
\begin{equation}
  \left\{
  \begin{aligned}
    (- \Delta            - \kappa^2)   u &= 0, && \text{in } \Omega, \\
    (\partial_{\bm{n}_i} - \mathscr{T})u &= 0, && \text{on } \Gamma^{\infty}_i, \\
    u &= -u^{\text{inc}}, && \text{on } \partial\Omega^{\text{sca}}, \\
  \end{aligned}
  \right.
  \label{eqn:origPbm}
\end{equation}
where $\kappa$ is the wavenumber, $u^{\text{inc}}$ is the incident wave, $\partial_{\bm{n}}$ is the exterior normal derivative, 
and $\mathscr{T}$ is an impedance operator to be defined.
We take the convention that the time-dependence of the fields is $e^{-\imath\omega t}$,
where $\omega$ is the angular frequency and $t$ is the time.

The domain decomposition method consists in considering the $N_{\text{dom}}$ local sub-problems coupled by the Robin conditions:
Seek the field $u_{\mathbf{i}}(\mathbf{x})$ that verifies
\begin{equation}
  \left\{
  \begin{aligned}
  (- \Delta -                \kappa^2      ) u_{\mathbf{i}} &= 0, && \text{in } \Omega_{\mathbf{i}}, \\
  ( \partial_{\bm{n}_{\mathbf{i},i}}          - \mathscr{T}) u_{\mathbf{i}} &= 0, && \text{on } \partial\Omega_{\mathbf{i}} \cap \Gamma^{\infty}_i, \\
  ( \partial_{\bm{n}_{\mathbf{i},\mathbf{j}}} - \mathscr{T}) u_{\mathbf{i}} &=
  (-\partial_{\bm{n}_{\mathbf{j},\mathbf{i}}} - \mathscr{T}) u_{\mathbf{j}} ,     && \text{on } \Gamma_{\mathbf{i},\mathbf{j}}, \forall \mathbf{j} \in D_{\mathbf{i}}, \\
  u_{\mathbf{i}} &= -u^{\text{inc}}, && \text{on } \partial\Omega_{\mathbf{i}}\cap\partial\Omega^{\text{sca}}, \\
  \end{aligned}
  \right.
  \label{eqn:subPbm}
\end{equation}
where the set
$
D_{\mathbf{i}} := \left\{ \mathbf{j} \, | \, \mathbf{j} \neq \mathbf{i} \text{ and } \Gamma_{\mathbf{i}, \mathbf{j}} \neq \varnothing \right\}.
$
%In this paper, we use high-order absorbing boundary conditions (HABCs) as transmission conditions, which provide an excellent approximation of the outgoing waves for both layered-type domain decompositions and checkerboard-type domain decompositions \cite{boubendir2012quasi,modave2020non}. 
%Please note that, when applying these conditions to polygonal domains, special treatment is required at corners in 2D cases to ensure a high-fidelity solution.
%In Section \ref{sec:interface_problem}, we employ HABCs. But, in the following section, where we investigate the algebraic structure of the interface problem, we use boundary conditions based on the basic impedance operator for the sake of clarity, even though they are less effective than the high-quality boundary conditions.
The paper uses high-order absorbing boundary conditions (HABCs) as transmission conditions, which are effective for both layered-type and checkerboard-type domain decompositions \cite{boubendir2012quasi,modave2020non}. 
However, special treatment is required at corners in 2D cases for polygonal domains. 
Section \ref{sec:interface_problem} of the paper employs HABCs, but in the next section, the paper uses less effective boundary conditions based on the basic impedance operator to investigate the algebraic structure of the interface problem for clarity.

\section{Interface problem}
\label{sec:interface_problem}

To derive the interface problem, let's introduce $w_{\mathbf{i}}(\mathbf{x})$ a lifting of the source: Seek $w_{\mathbf{i}}(\mathbf{x})$ that verifies
\begin{equation}
  \left\{
  \begin{aligned}
  (- \Delta - \kappa^2)                                      w_{\mathbf{i}} &= 0, && \text{in } \Omega_{\mathbf{i}}, \\
  ( \partial_{\bm{n}_{\mathbf{i},i}}          - \mathscr{T}) w_{\mathbf{i}} &= 0, && \text{on } \partial\Omega_{\mathbf{i}} \cap \Gamma^{\infty}_i, \\
  ( \partial_{\bm{n}_{\mathbf{i},\mathbf{j}}} - \mathscr{T}) w_{\mathbf{i}} &= 0, && \text{on } \Gamma_{\mathbf{i},\mathbf{j}}, \forall \mathbf{j} \in D_{\mathbf{i}}, \\
  u_{\mathbf{i}} &= -u^{\text{inc}}, && \text{on } \partial\Omega_{\mathbf{i}}\cap\partial\Omega^{\text{sca}}, \\
  \end{aligned}
  \right.
  \label{eqn:subPbmLifting}
\end{equation}
By the linearity of the problem, the field $u_{\mathbf{i}}$ can be decomposed into $v_{\mathbf{i}} + w_{\mathbf{i}}$, where $v_{\mathbf{i}}$ is the field \eqref{eqn:subPbm} after lifting the sources by \eqref{eqn:subPbmLifting}.
%\subsection{Linear system}
We introduce the local scattering operator
$
  \mathscr{S}_{m(\mathbf{j},\mathbf{i}),m(\mathbf{i},\mathbf{k})}: x_{m(\mathbf{i},\mathbf{k})} 
  \to (-\partial_{\bm{n}_{\mathbf{i},\mathbf{j}}} - \mathscr{T}) v_{\mathbf{i}}
$
where
\begin{equation}
  \left\{
  \begin{aligned}
    (- \Delta - \kappa^2)                                      v_{\mathbf{i}} &= 0,      && \text{in } \Omega_{\mathbf{i}}, \\
    (\partial_{\bm{n}_{\mathbf{i},i }}          - \mathscr{T}) v_{\mathbf{i}} &= 0,      && \text{on } \partial\Omega_{\mathbf{i}} \cap \Gamma^{\infty}_i, \\
    (\partial_{\bm{n}_{\mathbf{i},\mathbf{k} }} - \mathscr{T}) v_{\mathbf{i}} &= x_{m(\mathbf{i},\mathbf{k})}, && \text{on } \Gamma_{\mathbf{i},\mathbf{k}}, \\
    (\partial_{\bm{n}_{\mathbf{i},\mathbf{l} }} - \mathscr{T}) v_{\mathbf{i}} &= 0,      && \text{on } \Gamma_{\mathbf{i},\mathbf{l}}, \forall  \mathbf{l} \neq \mathbf{k},
  \end{aligned}
  \right.
\end{equation}
and $\mathbf{j},\mathbf{k},\mathbf{l} \in D_{\mathbf{i}}$.
Using the linearity of the problem and the above scattering operator,
we obtain the interface problem
$$
( \partial_{\bm{n}_{\mathbf{j},\mathbf{i}}} - \mathscr{T}) v_{\mathbf{j}} = \sum_{\mathbf{k} \in D_{\mathbf{i}}} \mathscr{S}_{m(\mathbf{j},\mathbf{i}),m(\mathbf{i},\mathbf{k})}
( \partial_{\bm{n}_{\mathbf{i},\mathbf{k}}} - \mathscr{T}) v_{\mathbf{i}} +
(-\partial_{\bm{n}_{\mathbf{i},\mathbf{j}}} - \mathscr{T}) w_{\mathbf{i}} , \quad \mathbf{j} \in D_{\mathbf{i}} .
$$
We introduce the global scattering matrix 
$S \in M_{2 N_e}(\mathscr{S}_{m(\mathbf{j},\mathbf{i}),m(\mathbf{i},\mathbf{k})})$,
the global additional variable vector
$g \in M_{2 N_e \times 1}(g_{m(\mathbf{j},\mathbf{i})})$,
and the global right-hand-side vector
$b \in M_{2 N_e \times 1}(b_{m(\mathbf{j},\mathbf{i})})$, where
$$
%\begin{aligned}
%& g_{m(\mathbf{j},\mathbf{i})} = (+\partial_{\bm{n}_{\mathbf{j},\mathbf{i}}} - \mathscr{T}) v_{\mathbf{j}}, \\
%& b_{m(\mathbf{j},\mathbf{i})} = (-\partial_{\bm{n}_{\mathbf{i},\mathbf{j}}} - \mathscr{T}) w_{\mathbf{i}}.
%\end{aligned}
g_{m(\mathbf{j},\mathbf{i})} = (+\partial_{\bm{n}_{\mathbf{j},\mathbf{i}}} - \mathscr{T}) v_{\mathbf{j}}, \quad
b_{m(\mathbf{j},\mathbf{i})} = (-\partial_{\bm{n}_{\mathbf{i},\mathbf{j}}} - \mathscr{T}) w_{\mathbf{i}}.
$$
We obtain that $g$ is the solution to the global matrix 
\begin{equation}
(I - S) g = b,
\label{eqn:matrix:abstractSys}
\end{equation}
or
\begin{equation}
g_{m(\mathbf{j},\mathbf{i})} - \displaystyle \sum_{\mathbf{k}} \mathscr{S}_{m(\mathbf{j},\mathbf{i}),m(\mathbf{i},\mathbf{k})} g_{m(\mathbf{i},\mathbf{k})} = b_{m(\mathbf{j},\mathbf{i})},
\quad \forall \mathbf{j}, \text{ for } \mathbf{i},\mathbf{k} \in D_{\mathbf{j}}.
\end{equation}

\section{Sweeping preconditioner}
\label{sec:sweeping}

Let $V_{i}$ be the $2N_e \times n_i$ matrix
%$$
%V_i = (e_{m(\mathbf{i},\mathbf{j})}), \quad |\mathbf{i}|_1 = i, \quad i = 2,\ldots,N_g+1,
%$$
$
V_i = (e_{m(\mathbf{i},\mathbf{j})}),
$
with 
$
|\mathbf{i}|_1 = i \, (i = 2,\ldots,N_g+1),
$
where each $e_{m(\mathbf{i},\mathbf{j})}$ is the $m(\mathbf{i},\mathbf{j})$-th 
column of the $2N_e \times 2N_e$ identity matrix, 
and $n_i$ is the number of columns.
One has $V_{i}^{\top} V_{j} = \mathbf{0}$, if $i \neq j$.  
Let $S_{i,j}$ be the $n_i \times n_j$ matrix
$
S_{i,j} = V_i^{\top} S V_j.
$

\begin{proposition} \label{proposition:1}
The upper and lower triangular matrix of the global matrix \eqref{eqn:matrix:abstractSys} 
can be decomposed by Gaussian elimination. 
\end{proposition}
\begin{proof}
We denote the lower triangular matrix of the global matrix \eqref{eqn:matrix:abstractSys} $S_L$. 
Consider the following matrices
$
\prod_i (I - V_i S_{i,i-1} V_{i-1}^{\top}), \quad i = 3,\ldots,N_g+1.
$
$\forall i = 3,\ldots,N_g$, we have
$$
\begin{aligned}
& (I - V_i S_{i,i-1} V_{i-1}^{\top}) (I - V_{i+1} S_{i+1,i} V_{i}^{\top}) 
\\
=
& \, I - V_i S_{i,i-1} V_{i-1}^{\top} - V_{i+1} S_{i+1,i} V_{i}^{\top} 
+ V_i S_{i,i-1} V_{i-1}^{\top} V_{i+1} S_{i+1,I} V_{i}^{\top}
\\
=
& \, I - V_i S_{i,i-1} V_{i-1}^{\top} - V_{i+1} S_{i+1,i} V_{i}^{\top} 
\end{aligned}
$$
The last term at the 2nd line vanishs 
since $V_{i-1}^{\top} V_{i+1}$ is null. 
Hence, we have
$$
  \prod_i (I - V_i S_{i,i-1} V_{i-1}^{\top})
= I - \sum_i   V_i S_{i,i-1} V_{i-1}^{\top}
= S_L.
$$
Similarily, we can proof that 
the upper triangular matrix of the global matrix \eqref{eqn:matrix:abstractSys} $S_U$ 
can be decomposed as
$$
S_U = \prod_i (I - V_{i-1} S_{i-1,i} V_{i}^{\top}), \quad i = N_g+1,\ldots,3,
$$
which is a series of matrices.
\end{proof}

Next, we present the Symmetric Gauss-Seidel (SGS) sweeping preconditioner $P_{\text{SGS}}$.
This matrix can then be rewritten as $P_{\text{SGS}} \approx S_L S_U$,
We can easily invert the matrix $P_{\text{SGS}}^{-1} = S_U^{-1} S_L^{-1}$,
with
%\begin{equation}
%  S_U^{-1}
%= \prod_i (I + V_{i-1} S_{i-1,i} V_{i}^{\top}), \quad i = 3,\ldots,N_g+1.
%\label{eqn:sgs_preconditioner_forward}
%\end{equation}
%and
%\begin{equation}
%  S_L^{-1}
%= \prod_i (I + V_i S_{i,i-1} V_{i-1}^{\top}), \quad i = N_g+1,\ldots,3,
%\label{eqn:sgs_preconditioner_backward}
%\end{equation}
%Observing Eqn. \eqref{eqn:sgs_preconditioner_forward} and Eqn. \eqref{eqn:sgs_preconditioner_backward}, we notice that it consists of a sequential process, in which there are $2 (N_g -1)$ sequential \textit{steps} in total. 
\begin{equation}
\begin{aligned}
& S_U^{-1}
= \prod_i (I + V_{i-1} S_{i-1,i} V_{i}^{\top}), \quad i = 3,\ldots,N_g+1, \\
& S_L^{-1}
= \prod_i (I + V_i S_{i,i-1} V_{i-1}^{\top}), \quad i = N_g+1,\ldots,3. \\
\end{aligned}
\label{eqn:sgs_preconditioner_forward_backward}
\end{equation}
Observing Eqn. \eqref{eqn:sgs_preconditioner_forward_backward}, we notice that it consists of a sequential process, in which there are $2 (N_g -1)$ sequential \textit{steps} in total. 

\section{Block Jacobi sweeping preconditioner}
\label{sec:block_jacobi_sweeping}

Let ${W_L}_i$, ${W_U}_i$ be the $2N_e \times s_i$ matrix 
$$
\begin{aligned}
& {W_L}_i = (e_{m(\mathbf{i},\mathbf{j})}), \quad  2+  (i-1)N_1 \leq |\mathbf{i}|_1 \leq  2+         i N_1, \\
& {W_U}_i = (e_{m(\mathbf{i},\mathbf{j})}), \quad (1+N_g)-i N_1 \leq |\mathbf{i}|_1 \leq (1+N_g)-(i-1) N_1, \\
\end{aligned}
$$
where $s_i$ is the number of columns ($e_{m(\mathbf{i},\mathbf{j})}$).
Let ${S_L}_i$, ${S_U}_i$ be the $s_i \times s_i$ matrix
$$
%\begin{aligned}
%& {S_L}_i = {W_L}_i^{\top} S {W_L}_i, \\
%& {S_U}_i = {W_U}_i^{\top} S {W_U}_i. \\
%\end{aligned}
{S_L}_i = {W_L}_i^{\top} S_L {W_L}_i, \quad 
{S_U}_i = {W_U}_i^{\top} S_U {W_U}_i.
$$
According to the \textit{aditive projection processes} \cite{saad2003iterative}, 
the next iterate can be defined as
$$
\begin{aligned}
& x^{(k+1/2)} = x^{(k)}     + \sum_{i}^p {W_L}_i (I_i + {S_L}_i)^{-1} {W_L}_i^{\top} r^{(k)}    , \\
& x^{(k+1)}   = x^{(k)+1/2} + \sum_{i}^p {W_U}_i (I_i + {S_U}_i)^{-1} {W_U}_i^{\top} r^{(k+1/2)}. \\
\end{aligned}
$$
In the equations above, we call $I_i + {S_L}_i$ ($I_i + {S_U}_i$) the forward (backward) block Jacobi preconditioner 
which can be decomposed into a serie of matrices according to the Proposition \ref{proposition:1}.
The forward (backward) block Jacobi preconditioner consists of block diagonals of $S$, 
which can improve the parallel performance of the sweeping preconditioner.
Consequently, there are $2N_1$ sequential steps in the forward and backward block Jacobi preconditioners.

\section{Numerical results}
\label{sec:numerical_results}

\begin{figure}[!tb]
  \centering
  \small
  \begin{subfigure}[!tb]{0.225\textwidth}
  \includegraphics[scale = 0.110]{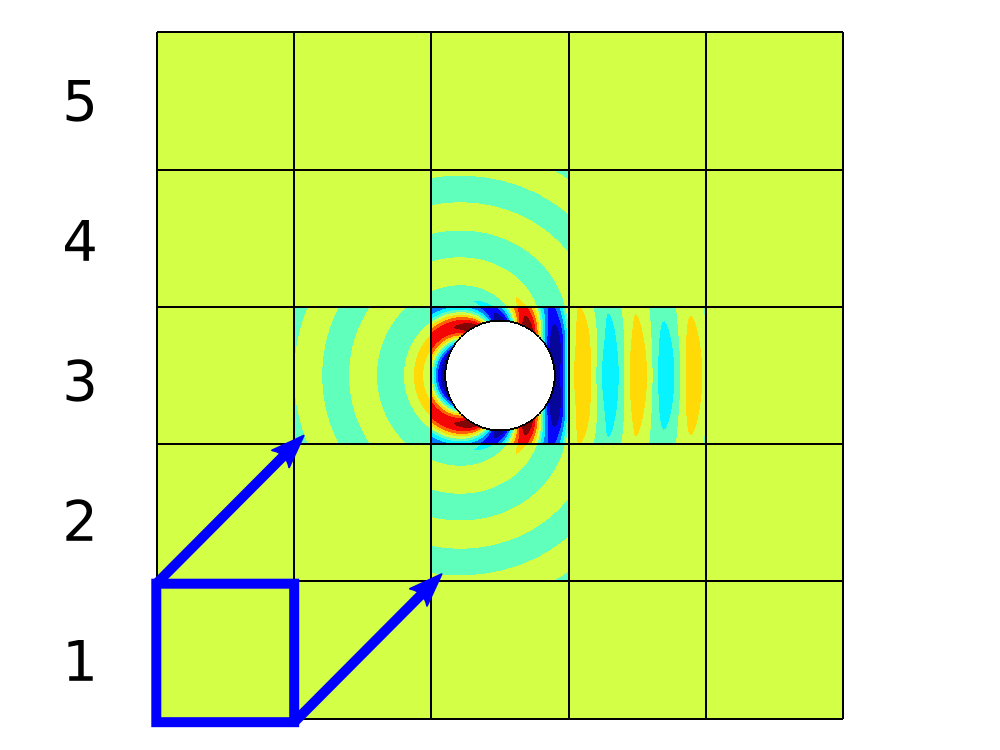} 
  \caption{1st step}
  \end{subfigure}
  \begin{subfigure}[!tb]{0.225\textwidth}
  \includegraphics[scale = 0.110]{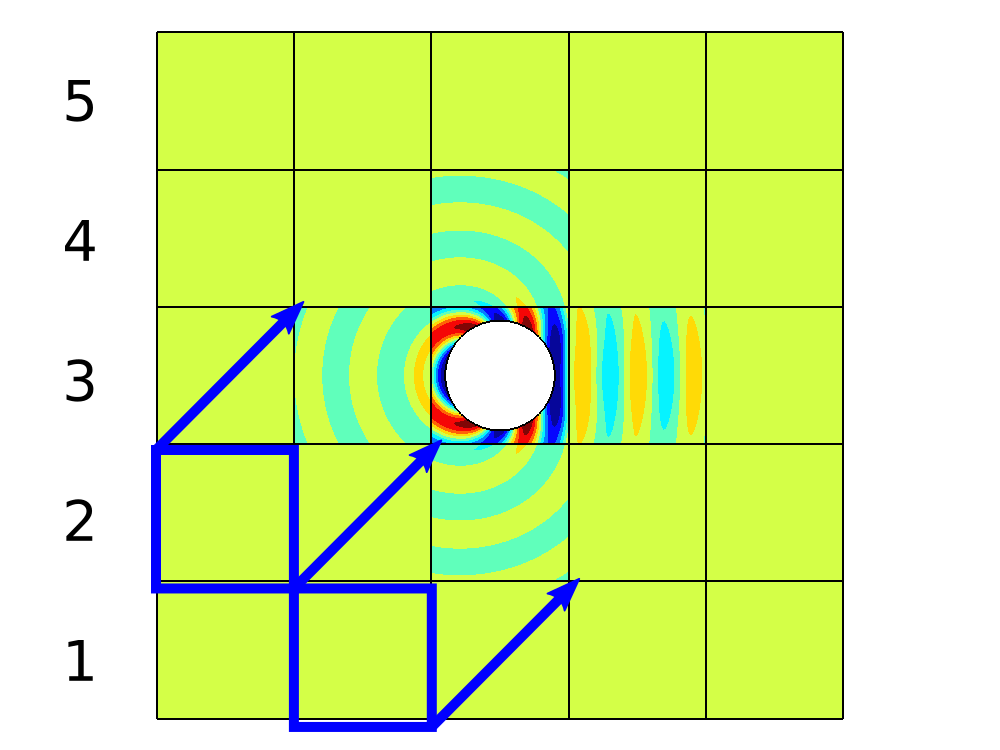} 
  \caption{2nd step}
  \end{subfigure}
  \begin{subfigure}[!tb]{0.225\textwidth}
  \includegraphics[scale = 0.110]{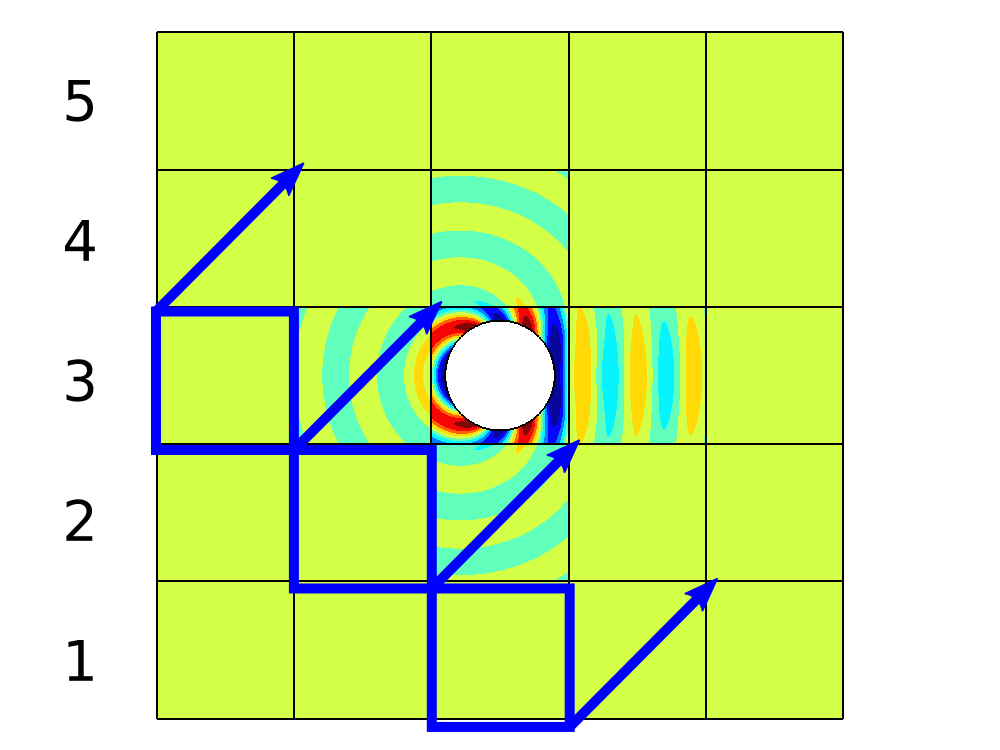} 
  \caption{3rd step}
  \end{subfigure}
  \begin{subfigure}[!tb]{0.225\textwidth}
  \includegraphics[scale = 0.110]{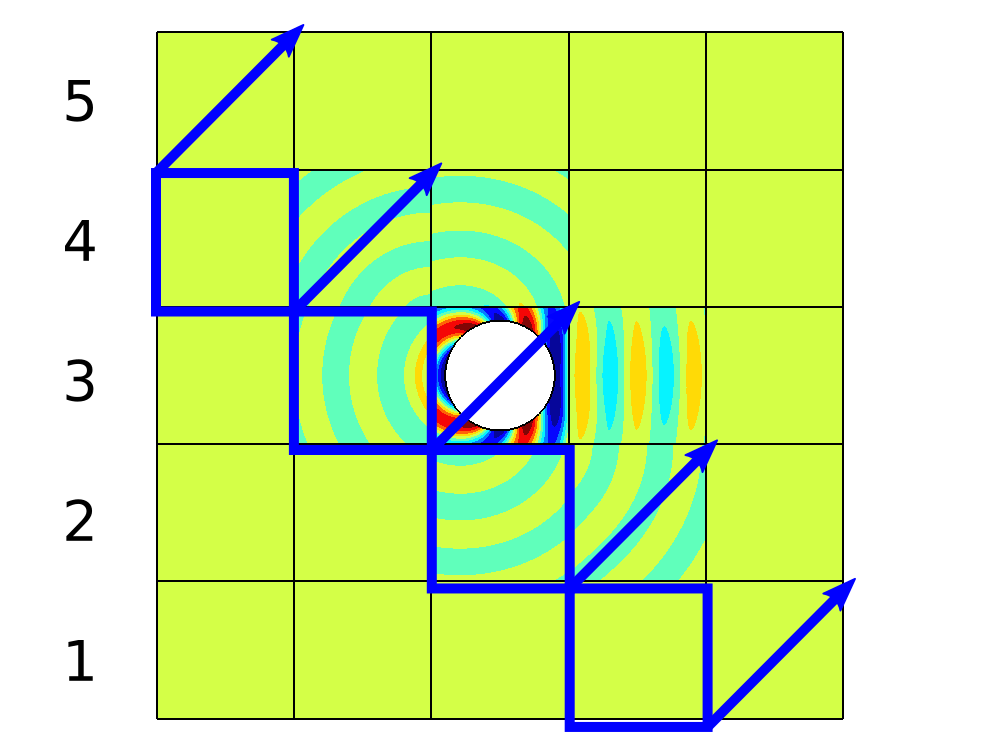} 
  \caption{4th step}
  \end{subfigure}
  \begin{subfigure}[!tb]{0.225\textwidth}
  \includegraphics[scale = 0.110]{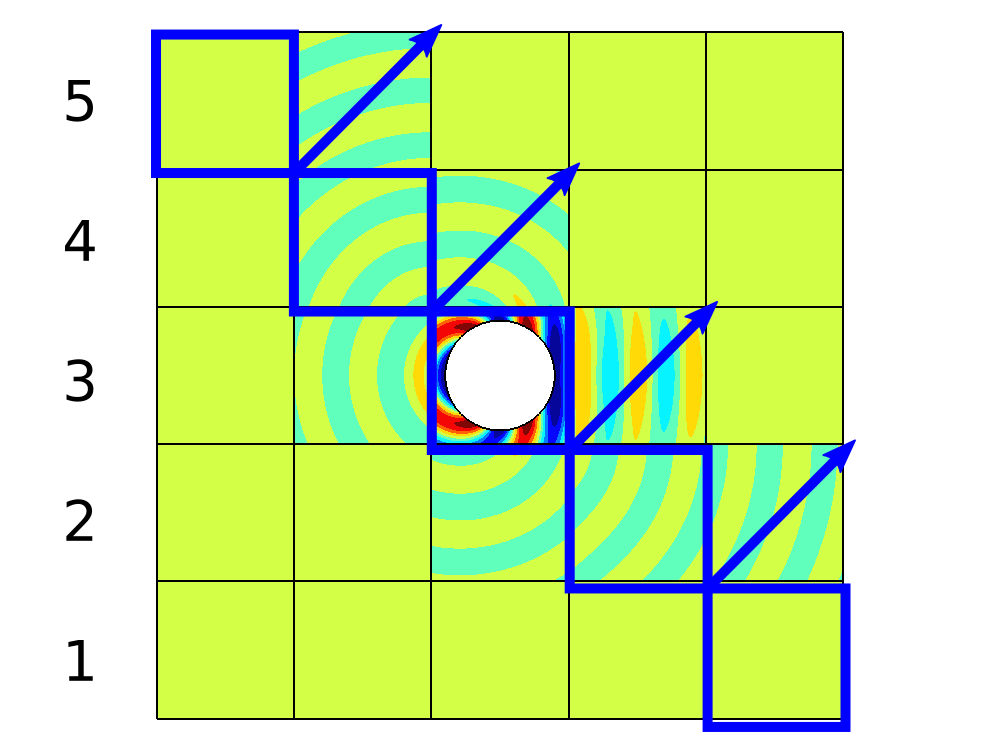} 
  \caption{5th step}
  \end{subfigure}
  \begin{subfigure}[!tb]{0.225\textwidth}
  \includegraphics[scale = 0.110]{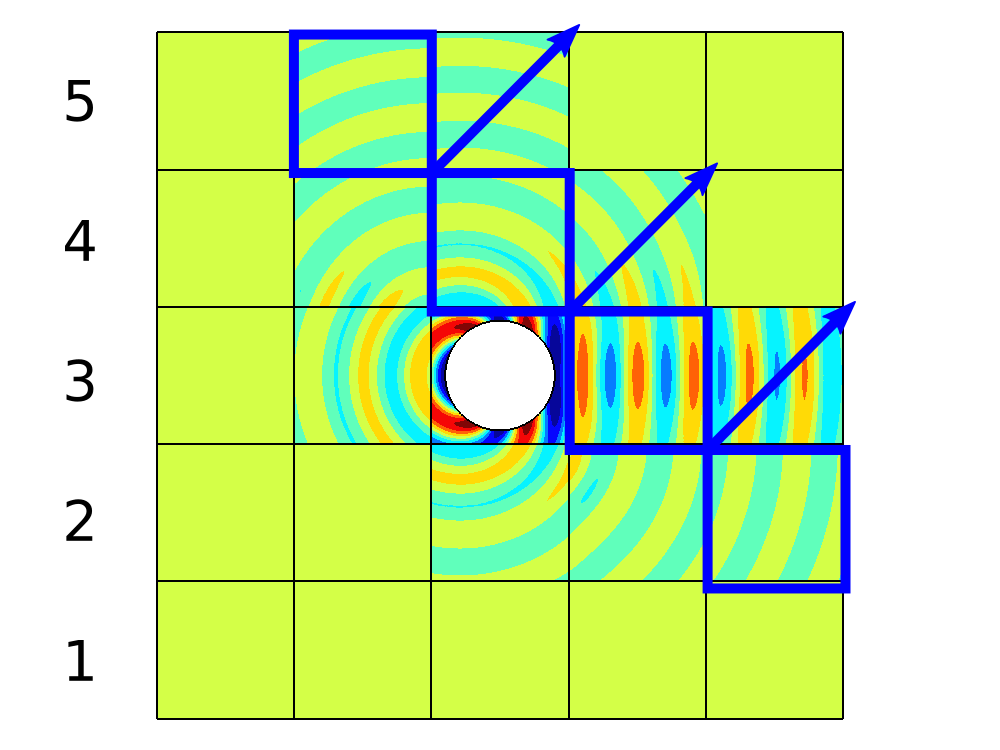} 
  \caption{6th step}
  \end{subfigure}
  \begin{subfigure}[!tb]{0.225\textwidth}
  \includegraphics[scale = 0.110]{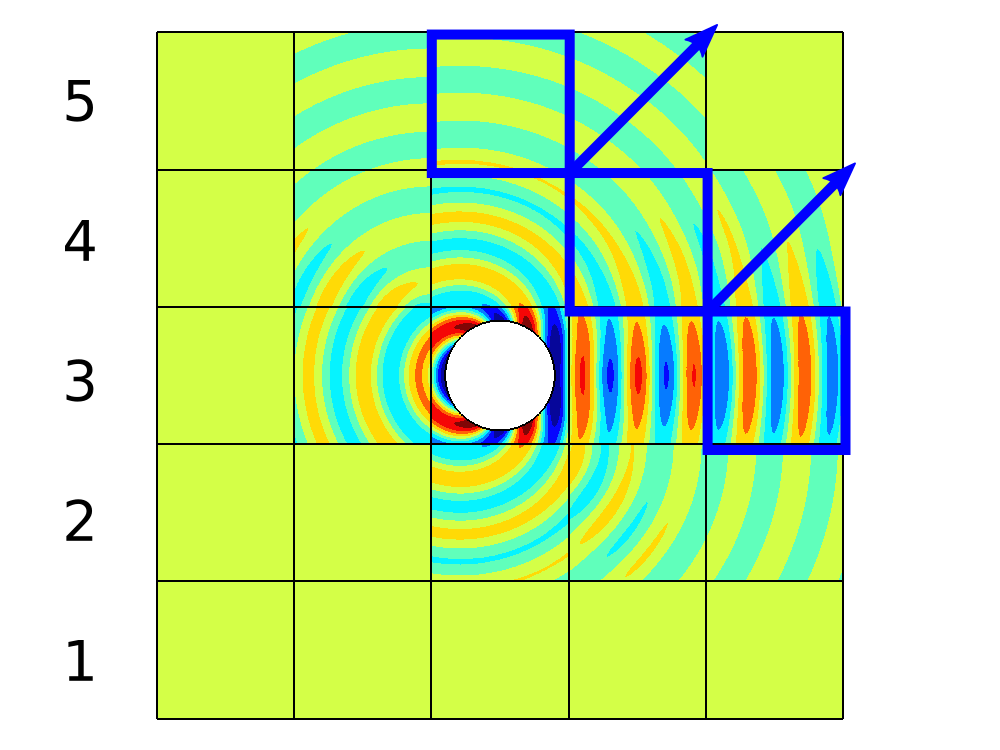} 
  \caption{7th step}
  \end{subfigure}
  \begin{subfigure}[!tb]{0.225\textwidth}
  \includegraphics[scale = 0.110]{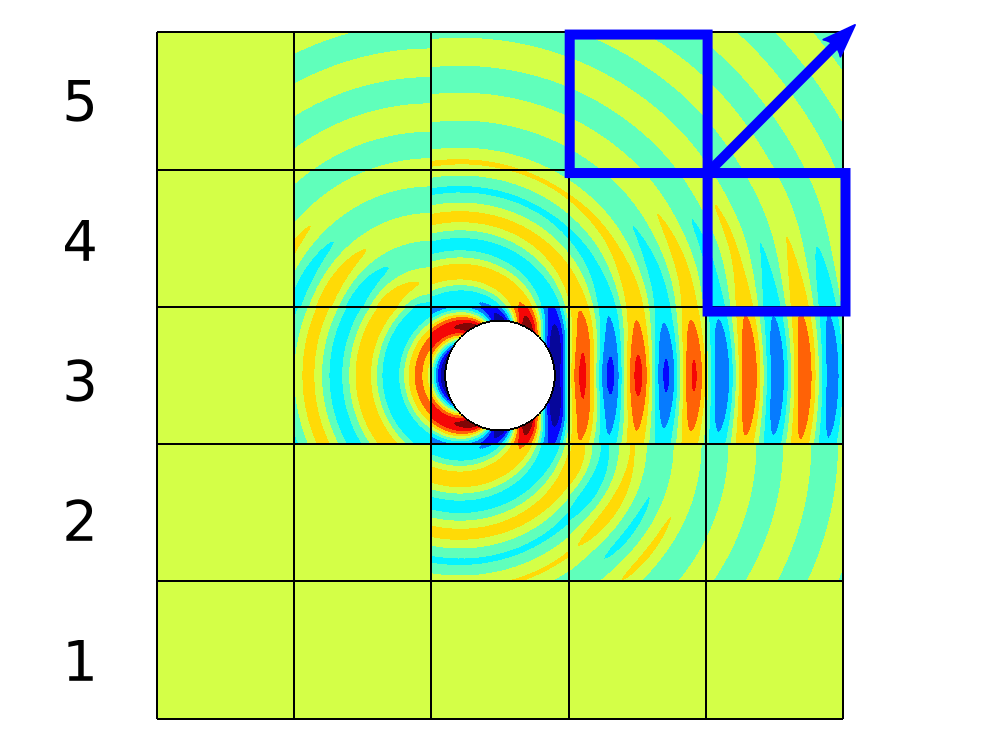} 
  \caption{8th step}
  \end{subfigure}
 %\begin{subfigure}[!tb]{0.30\textwidth}
 %\includegraphics[scale = 0.150]{images/ps_final/res_scatter_f9.png} 
 %\caption{Visu after forward sweep}
 %\end{subfigure}
  \captionsetup{font=small, labelfont=bf}
 %\caption{Scattering model in 2D ($k=2\pi$). Snapshot of the solution at different steps of forward sweep of 1st GMRES iteration with the full sweeping preconditioner. The numbers at left side are processors' identities. Each row of subdomains is assigned to one MPI rank. Subdomains processed in parallel have same blue color.}
  \caption{Scattering model in 2D ($k=2\pi$) with a snapshot of the solution at different steps of the first GMRES iteration using the full sweeping preconditioner. Each row of subdomains is assigned to one MPI rank, and processors are identified by the numbers on the left side. Subdomains processed in parallel are highlighted in blue.}
  \label{snapshot:steps_of_full_forward_sweep}
\end{figure}

In this part, the block Jacobi sweeping preconditioner is studied 
by considering a two-dimensional benchmark with a high-order finite element method. 
The proposed approaches and the computational results presented in this paper are implemented in parallel by MPI on a single multi-core computer. 
The linear systems arising from the sub-problems are solved by a sparse direct solver.
%PARDISO \cite{pardiso-7.2a,pardiso-7.2b,pardiso-7.2c}.
The mesh generation, mesh decomposition, and post-processing are credited by Gmsh \cite{geuzaine2009gmsh}.
%Implementing the approach using MPI, we assign each row or column of subdomains to one MPI rank. In terms of preconditioned optimized Schwarz methods, each local sub-problem is solved twice at each iteration of GMRES: the one at the level of matrix-free product, the other at the level of preconditioning procedure. If we assign each subdomain to one MPI rank, the sequential nature of sweeping preconditioners limits the efficiency of the approach. 
%At the level of matrix-free product, every MPI rank can solve its local sub-problem concurrently within one iteration. 
%At the level of the preconditioning procedure, sweeps must be done in a sequence of steps, which, from a time measurement perspective, 
%corresponds to $2(N_1 + N_2 - 1)$ iterations in two-dimensional cases and  $2(N_1 + N_2 + N_3 - 1)$ iterations. 
%Thus, within the scope of the preconditioning procedure, 
%there are a lot of wasted computing resources and increased computing time 
%if one chooses to assign each local problem to one MPI rank.
%On the contrary, if we assign each row or column of subdomains to one MPI rank, 
%the steps of traversing all subdomains at the level of the matrix-free product are equivalent to the steps of the sweeping preconditioning procedure, 
%which allows one to obtain a parallel scalable preconditioned solver.
The parallelism of our approach is realized by assigning subdomains to MPI ranks in a row-based fashion such that the $i$-th row of the checkerboard domain decomposition is processed by rank $i$.

The test case is a homogeneous scattering problem in free space within a rectangle geometry ($\Omega = [-1.25, 2.50 \cdot N_1 - 1.25] \times [-1.25, 2.50 \cdot N_2 - 1.25]$), 
which is decomposed into $N_1 \times N_2$ rectangular subdomains.
An incident plane wave is generated by a sound-soft circular cylinder of radius equal to $1$ which is located at the Origin.
On the circular cylinder, the Dirichlet boundary condition $u(\mathbf{x}) = - \exp^{i k x}$ is prescribed at the boundary of the sound-soft scatterer.
The Pad\'e-type HABC is prescribed on the exterior boundaries and the interior interfaces used as the absorbing boundary conditions and the transmission boundary conditions, respectively.
The compatibility conditions are prescribed at the corners and the cross-points treatment is prescribed at the cross-points. 
%We employ the Pad\'e-type HABC operator with a large number of auxiliary fields, which improves the efficiency of the DD methods for scattering problems, to show distinct residual history of GMRES.
The parameters of the HABC operator are $N_{\text{pade}} = 8$ and $\phi = \pi/3$. 
%We use a standard high-order finite element method to solve this problem. 
%The meshes used by the method are made of high-order triangles which are generated by Gmsh \cite{geuzaine2009gmsh}. 
The following numerical setting are considered: $P7$ finite elements with 3 elements per wavelength ($h \approx 1/21$). 
%The meshes of three partitions are made of 
%99023  nodes, 4016  $P7$ triangles (for $N_1 \times N_2 = 5 \times 5$  partition), 
%198479 nodes, 8064  $P7$ triangles (for $N_1 \times N_2 = 5 \times 10$ partition), and
%297935 nodes, 12112 $P7$ triangles (for $N_1 \times N_2 = 5 \times 15$ partition), repectively.
  
Fig. \ref{snapshot:steps_of_full_forward_sweep} and \ref{snapshot:steps_of_partial_forward_sweep} show snapshots of the solutions at different steps of forward sweep (sweep starts from the bottom-left corner to the top-right corner) of the 1st GMRES iteration with different sweeping preconditioners.
%The solution is already good at each step for these two preconditioners. 
%The HABC operator with a large number of auxiliary fields used in the transmission conditions gives high-fidelity values at transmission boundaries for sub-problems in sub-domains.
Although the forward sweep in Fig. \ref{snapshot:steps_of_full_forward_sweep} goes through the whole computational domain from the bottom-left corner to the top-right corner, it takes $8$ steps. If we take the backward sweep into account, there are $16$ steps of the preconditioning procedure at each iteration. In the second situation, it only takes $5$ steps in the forward partial sweeps (see Fig. \ref{snapshot:steps_of_partial_forward_sweep}). 
Fig. \ref{conv_test:2cut_scattering} shows snapshots of the solutions and residual histories of GMRES with the different preconditioners for the partition $N_1 \times N_2 = 5 \times 10$. 
All forward/backward (partial) sweeps of these preconditioners start from the bottom-left/top-left to the top-right/bottom-left.
The violet boxes indicate the cut location which separates partial sweeps. 

\begin{figure}[!tb]
  \centering
  \small
  \begin{subfigure}[!tb]{0.225\textwidth}
  \includegraphics[scale = 0.110]{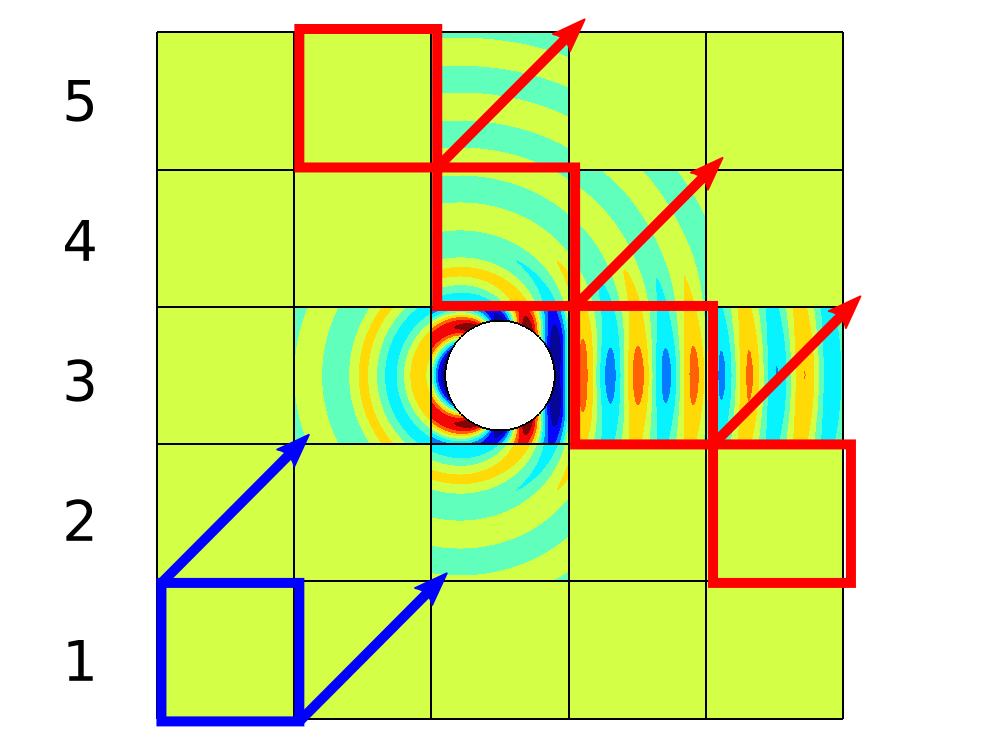} 
  \caption{1st step}
  \end{subfigure}
  \begin{subfigure}[!tb]{0.225\textwidth}
  \includegraphics[scale = 0.110]{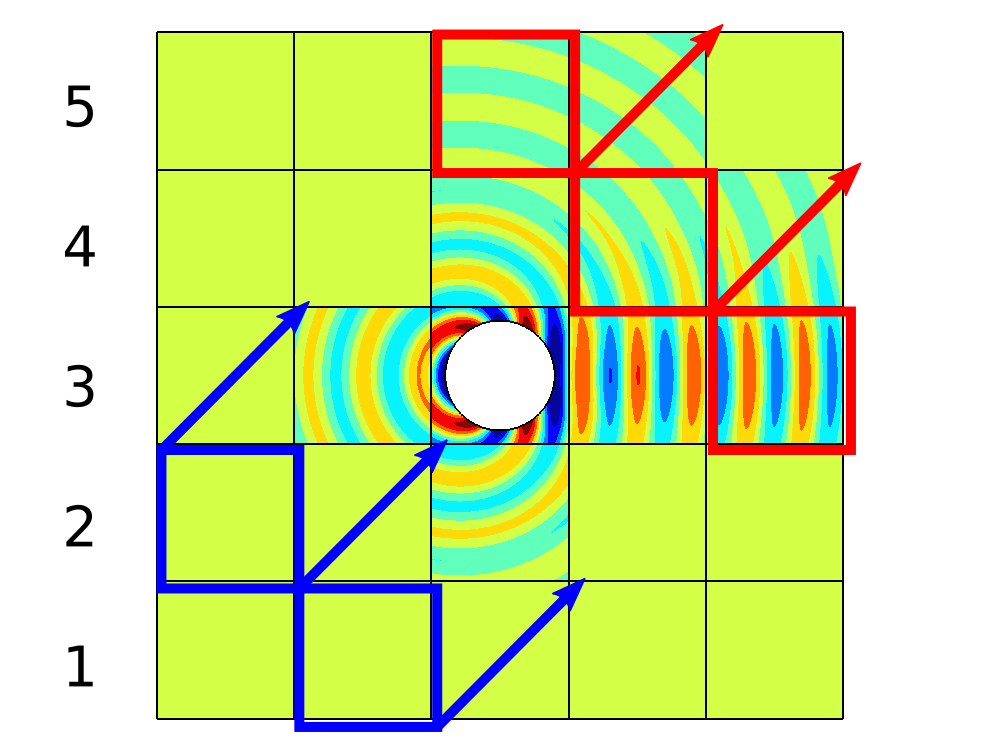} 
  \caption{2nd step}
  \end{subfigure}
  \begin{subfigure}[!tb]{0.225\textwidth}
  \includegraphics[scale = 0.110]{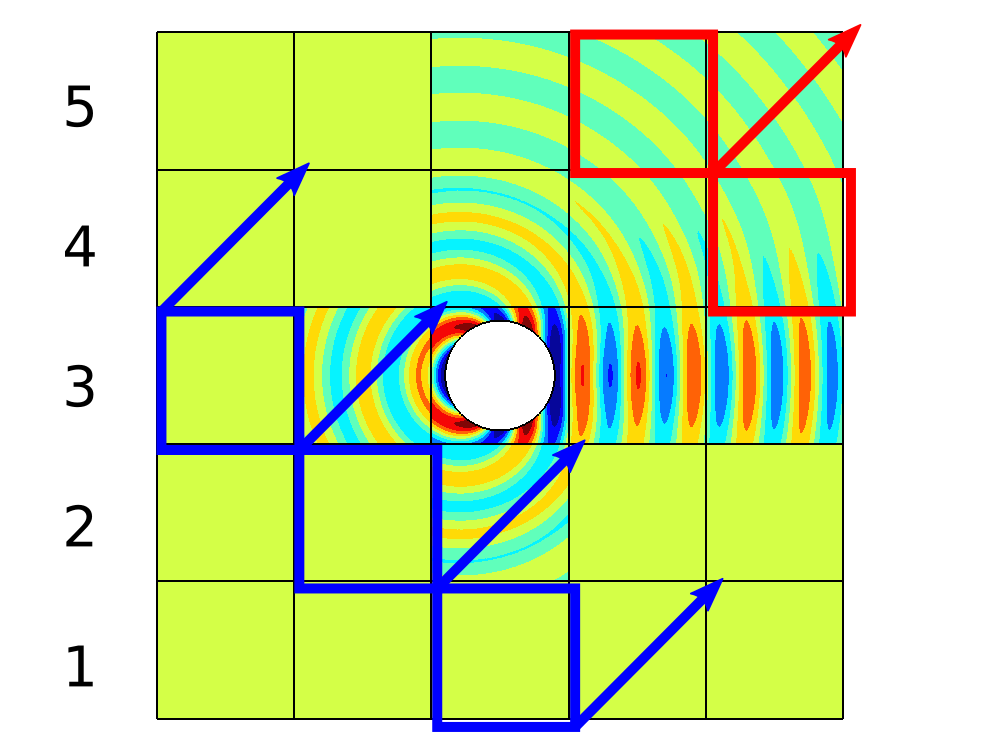} 
  \caption{3rd step}
  \end{subfigure}
  \begin{subfigure}[!tb]{0.225\textwidth}
  \includegraphics[scale = 0.110]{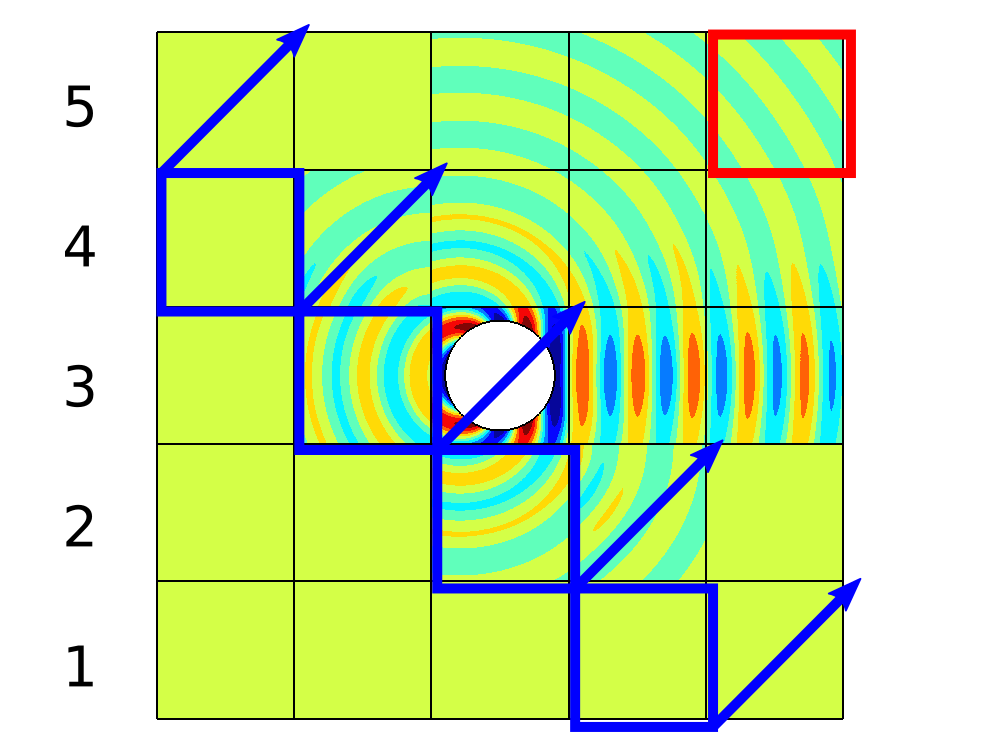} 
  \caption{4th step}
  \end{subfigure}
  \begin{subfigure}[!tb]{0.225\textwidth}
  \includegraphics[scale = 0.110]{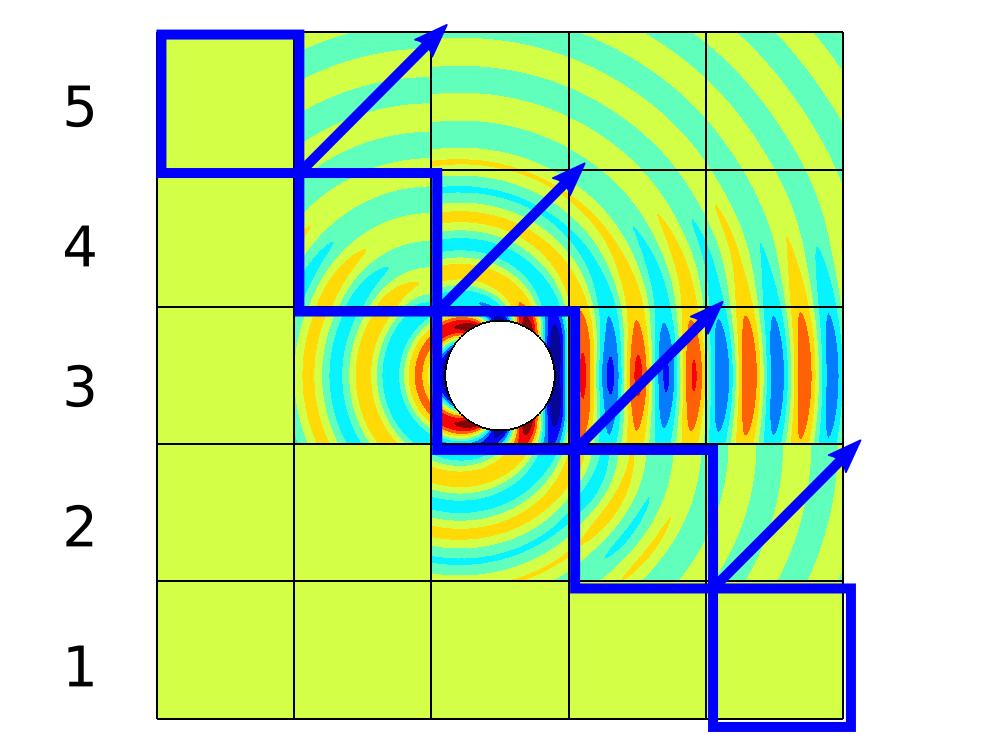} 
  \caption{5th step}
  \end{subfigure}
  %\begin{subfigure}[!tb]{0.30\textwidth}
  %\includegraphics[scale = 0.150]{images/psc_final/res_scatter_f6.png} 
  %\caption{Visu after forward sweep}
  %\end{subfigure}
  \captionsetup{font=small, labelfont=bf}
  \caption{Scattering model in 2D ($k=2\pi$). Snapshot of the solution at different steps of forward sweep of 1st GMRES iteration with the block Jacobi sweeping preconditioner. The numbers at left side are processors' identities. Each row of subdomains is assigned to one MPI rank. Subdomains processed in parallel have same blue and red color, which represent two partial sweeps.}
  \label{snapshot:steps_of_partial_forward_sweep}
\end{figure}

\begin{figure}[!tb]
  \centering\small
  \includegraphics[height=4.5cm]{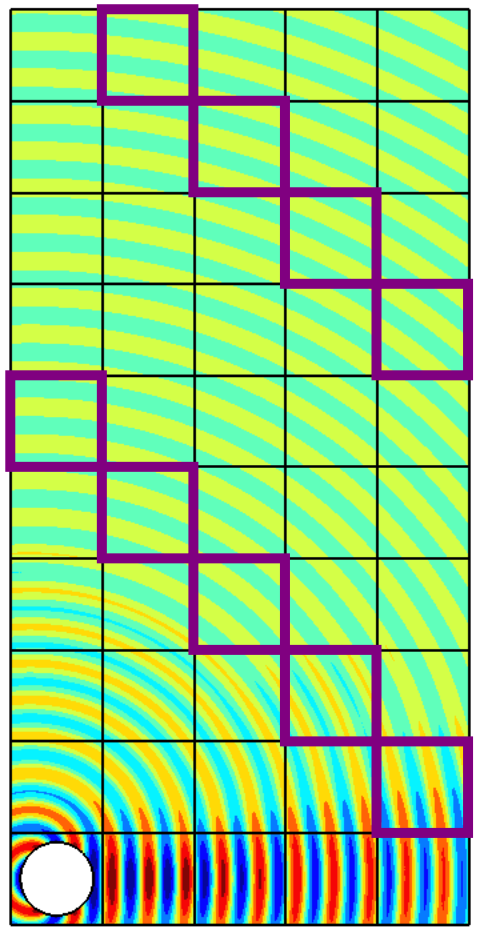} \hspace{3.20cm}
  \includegraphics[height=4.5cm]{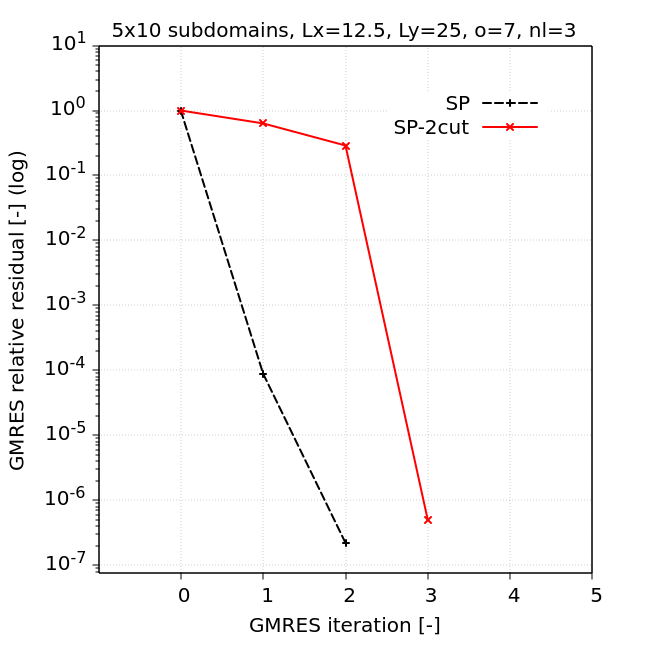}
  \caption{Scattering model in 2D ($k=2\pi$). The computational domain is decomposed into $N_1 \times N_2 = 5 \times 10$. Residual history with sweeping preconditioner (SP) and block Jacobi preconditioner (BSP) with two cuts.}
  \label{conv_test:2cut_scattering}
\end{figure}

%\begin{figure}[!tb]
%  \centering\small
%  \includegraphics[height=4.5cm]{images/conv_test/3cut_visu.png} \hspace{3.95cm}
%  \includegraphics[height=4.5cm]{images/conv_test/3cut_residual.png}
%  \caption{Scattering model in 2D. The computational domain is decomposed into $N_1 \times N_2 = 5 \times 15$. Residual history with sweeping preconditioner (SP) and block Jacobi preconditioner with one cut (SP-3cut).}
%  \label{conv_test:3cut_scattering}
%\end{figure}

%Fig. \ref{conv_test:2cut_scattering} shows snapshots of the solutions and residual histories of GMRES with the different preconditioners for three partitions. 
%All forward/backward (partial) sweeps of these preconditioners start from the bottom-left/top-left to the top-right/bottom-left.
%The violet boxes indicate the cut location which separates partial sweeps. 
%For instance, in Fig. \ref{conv_test:1cut_scattering}, two forward partial sweeps in the block Jacobi preconditioner start from the bottom-left corner to the violet boxes and from the violet boxes to the top-right corner, respectively.

\begin{table}[!t]
\caption{Scattering model in 2D ($k=20\pi$). Number of iterations and runtime in seconds with the two different preconditioners for different domain partitions. ``ni'' stands for the number of iterations, ``ns'' the number of steps per iteration, and ``t'' time. The number of MPI ranks is equal to $N_2$.}
\label{tab:1}       % Give a unique label
%
% Follow this input for your own table layout
%
\begin{tabular}{p{2.0cm}p{1.5cm}p{1.5cm}p{1.5cm}p{1.5cm}p{1.5cm}p{1.5cm}}
\hline\noalign{\smallskip}
$N_1 \times N_2$ & SP (ni) & SP (ns) & SP (t) & BSP (ni) & BSP (ns) & BSP (t)  \\
\noalign{\smallskip}\svhline\noalign{\smallskip}
$5 \times 5    $ & 3       & 16      & 32.6 s & 3        & 10       & 25.4 s   \\
$5 \times 10   $ & 3       & 26      & 49.0 s & 4        & 10       & 33.0 s   \\
$5 \times 15   $ & 4       & 36      & 90.8 s & 6        & 10       & 51.3 s   \\
$5 \times 20   $ & 5       & 46      &147.4 s & 8        & 10       & 70.9 s   \\
\noalign{\smallskip}\hline\noalign{\smallskip}
\end{tabular}
\end{table}

The residual histories obtained with the two different preconditioners 
in Fig. \ref{conv_test:2cut_scattering}, 
where the relative residual suddenly drops in residual history at the first iteration when a full sweeping preconditioner is used. With the block Jacobi sweeping preconditioner used, it happens at the third iteration, which corresponds to the number of partial sweeps, that is to say, there are two partial sweeps. 
%If Fig. \ref{conv_test:2cut_scattering} and \ref{conv_test:3cut_scattering}, the histories of the relative residual are similar. With two cuts, which corresponds to three partial sweeps, it takes two more iterations to reach the sudden drop compared to the history with the full sweeping preconditioner used. 
%Similarly, with three cuts (four partial sweeps), it takes three more iterations to reach the sudden drop.

The number of GMRES iterations and the runtime to reach a relative residual $10^{-6}$ 
with the two different preconditioners are given in Table \ref{tab:1}. 
The runtime corresponds to the GMRES resolution phase. 

\end{document}